\documentclass[12pt,a4paper,reqno]{amsart}
\usepackage{amssymb,a4wide,times,mathptmx,hyperref}
\providecommand{\BBb}[1]{{\mathbb{#1}}}

\providecommand{\cal}[1]{{\mathcal{#1}}}   

\newcommand{\Bcirc}{\overset{\lower 1.5pt%
              \hbox{$@,@,@,@,@,\scriptscriptstyle\circ$}}B{}}
\newcommand{\Binfty}{\overset{\lower 1.5pt%
              \hbox{$@,@,@,@,@,\scriptscriptstyle\infty$}}B{}}
\newcommand{\bigdot}{\mathbin{\raise.65\jot\hbox{$\scriptscriptstyle\bullet$}}}

\newcommand{\C}{{\BBb C}}

\newcommand{\dual}[2]{\langle\,#1,\,#2\,\rangle}

\newcommand{\erd}{\overset{\lower 1pt\hbox{\large.}}{e}
                  \overset{\lower 1pt\hbox{\large.}}{r}}

\newcommand{\Fcirc}{\overset{\lower 1.5pt%
               \hbox{$@,@,@,@,@,\scriptscriptstyle\circ$}}F{}}
\newcommand{\fracc}[2]{{
                \textstyle\frac{#1}{\raise 1pt\hbox{$\scriptstyle #2$}}}}

\newcommand{\fracnp}{\fracc np}

\newcommand{\fracci}[2]{{\frac{#1}{\raise 1pt\hbox{$\scriptscriptstyle #2$}}}}
\newcommand{\fracpi}{\fracci1p}

\newcommand{\im}{\operatorname{i}}

\newcommand{\loc}{\operatorname{loc}}

\newcommand{\nrm}[2]{\|#1\|_{#2}}
\newcommand{\Nrm}[2]{\bigl\|#1\bigr\|_{#2}}

\newcommand{\order}{\operatorname{order}}
\newcommand{\op}[1]{\operatorname{#1}}
\newcommand{\OP}{\operatorname{OP}}
\newcommand{\N}{\BBb N}

\renewcommand{\Re}{\operatorname{Re}}

\newcommand{\R}{{\BBb R}}
\newcommand{\Rn}{{\BBb R}^{n}}

\providecommand{\rom}[1]{\upn{#1}}

{
\end{minipage}%
\par\medskip\noindent}%
\newcounter{enmcount}\renewcommand{\theenmcount}{{\rm\arabic{enmcount}}}

\newcounter{rmcount}\renewcommand{\thermcount}{{\rm\roman{rmcount}}}
\newenvironment{rmlist}{%
\begin{list}{{\rm(\thermcount)}}{\setlength{\labelwidth}{\leftmargin}%
\usecounter{rmcount}}}{\end{list}}
\newcounter{Rmcount}\renewcommand{\theRmcount}{{\rm\Roman{Rmcount}}}

\newcommand{\Set}[2]{\bigl\{\,#1\bigm| #2\,\bigr\}}

\newcommand{\supp}{\operatorname{supp}}
\newcommand{\Z}{\BBb Z}

\renewcommand{\check}[1]{\overset{{\scriptscriptstyle \vee}}{#1}}
\renewcommand{\hat}[1]{\overset{{\scriptscriptstyle \wedge}}{#1}}

\numberwithin{equation}{section}

\newtheorem{thm}{Theorem}
\numberwithin{thm}{section}
\newtheorem{prop}[thm]{Proposition}
\newtheorem{lem}[thm]{Lemma}
\newtheorem{cor}[thm]{Corollary}

\theoremstyle{definition}
\newtheorem{defn}[thm]{Definition}

\newtheorem{exmp}[thm]{Example}
 \numberwithin{exercise}{section}
 
\theoremstyle{remark}
\newtheorem{rem}[thm]{Remark}

\title[$L_{{p}}$-Theory of Type $1,1$-Operators]{%
$\mathbf{L}_{\mathbf{p}}$-Theory of Type $\mathbf{1},\mathbf{1}$-Operators}
\author{Jon Johnsen}
\address{Department of Mathematical Sciences\\ Aalborg University\\ 
Fredrik Bajers Vej 7G\\ DK-9220 Aalborg {\O}st, Denmark}
\email{jjohnsen@math.aau.dk}
\subjclass[2000]{35S05,47G30}
\keywords{Exotic pseudo-differential operators, type $1,1$, twisted diagonal
condition, paradifferential decomposition, Spectral Support Rule,
factorisation inequality, corona conditions}
\dedicatory{Dedicated to Professor Hans Triebel on the Occasion of his Seventy Fifth Birthday}
\thanks{Supported by the Danish Council for Independent Research, Natural Sciences  
(Grant No. 09-065927)%
\\[5\jot]
{\tt Appeared in  Mathematische Nachrichten, {\bf 286} (2013), 712--729}} 
\begin{document}
\begin{abstract}
This is a continuation of recent work on the general definition of
pseudo-differential operators of type $1,1$, in H{\"o}rmander's sense.
Continuity in $L_p$-Sobolev spaces and
H{\"o}lder--Zygmund spaces, and more generally in Besov and
Lizorkin--Triebel spaces, is proved for positive smoothness;
with extension to arbitrary
smoothness for operators in the self-adjoint subclass. As a main tool the
paradifferential decomposition is used for type $1,1$-operators
in combination with the Spectral Support Rule for pseudo-differential
operators and pointwise estimates in terms of maximal functions of
Peetre--Fefferman--Stein type. 
\end{abstract}
\maketitle

\section{Introduction}  \label{intro-sect}
The understanding of pseudo-differential operators of type $1,1$
and their applications developed crucially in the 1980's through works of
Meyer \cite{Mey81}, Bony \cite{Bon}, 
Bourdaud \cite{Bou83,Bou88}, 
H{\"o}rmander \cite{H88,H89}; 
cf also the revised exposition in \cite[Ch.~9]{H97}.
Their theory was taken up again more recently by the author,
who showed 
that Lizorkin--Triebel spaces $F^s_{p,q}$ are optimal for certain borderlines
\cite{JJ04DCR,JJ05DTL}. 

However, the first general definition of type $1,1$-operators was
given in 2008 by the author in \cite{JJ08vfm} and used there in a discussion of
unclosability, pseudo-locality,
non-preservation of wavefront sets and the Spectral Support Rule.
The present paper continues the work in \cite{JJ08vfm} with a
systematic approach to their $L_p$-theory.

\bigskip

Recall that
by definition, the symbol $a(x,\eta)$ of a type $1,1$-operator of order
$d\in \R$ fulfils
\begin{equation}
  |D^\alpha_\eta D^\beta_x a(x,\eta)|\le
  C_{\alpha,\beta}(1+|\eta|)^{d-|\alpha|+|\beta|}\quad\text{for}\quad
  x,\eta\in \Rn.  
  \label{Cab-ineq}
\end{equation}
The corresponding operator is for \emph{Schwartz functions}, ie for $u\in \cal
S(\Rn)$, given by 
\begin{equation}
  a(x,D)u=(2\pi)^{-n}\int e^{-\im x\cdot \eta}a(x,\eta)\hat u(\eta)\,d\eta  .
\end{equation}
But a general definition for $u\in \cal S'\setminus \cal S$ must
take into account that in some cases they can
only be defined on proper subspaces $E\subset \cal S'$.

A rigorous definition of type $1,1$-operators was
first given in \cite{JJ08vfm}. 
Indeed, it was proposed to stipulate that $u$ belongs to the domain $D(a(x,D))$ and to set
\begin{equation}
  a(x,D)u:= \lim_{m\to\infty } (2\pi)^{-n}\int_{\Rn}
  e^{\im x\cdot\eta}\psi(2^{-m}D_x)a(x,\eta)\psi(2^{-m}\eta)\hat u(\eta)\,d\eta
  \label{a11-id}
\end{equation}
whenever this limit exists in $\cal D'(\Rn)$ for all the
$\psi\in C^\infty_0(\Rn)$ with $\psi=1$ in a neighbourhood of the origin and
does not depend on such $\psi$. 
(More precisely, one should replace the integral above by the action of
$\OP(\psi(2^{-m}D_x)a(x,\eta)\psi(2^{-m}\eta))$ in $\OP(S^{-\infty})$ on $u$.)

This unconventional definition, by 
\emph{vanishing frequency modulation}, was motivated by the applications of type
$1,1$-operators in the theory of semi-linear elliptic boundary
problems in the author's work \cite{JJ08par}.

In the present paper, the main question is to obtain boundedness
\begin{equation}
  \nrm{a(x,D)u}{s}\le c_s \nrm{u}{s+d},
\label{s-ineq}
\end{equation}
where the $s$-dependent norms can be those of the Sobolev spaces
$H^s_p$ (for a fixed $p\in\,]1,\infty[\,$), the
H{\"o}lder--Zygmund spaces $C^s_*$, or even of the
Besov spaces $B^s_{p,q}$ or Lizorkin--Triebel spaces $F^s_{p,q}$.
These $L_p$-results constitute an important justification of the
definition in \eqref{a11-id}.

The proofs are based on Littlewood--Paley theory,
where it has been most useful to adopt the
pointwise estimates in the recent article
\cite{JJ11pe}. Indeed, this gives the \emph{factorisation inequality}
\begin{equation}
  |a(x,D)u(x)|\le F_a(x)\cdot u^*(x)
\label{apest-eq}
\end{equation}
in terms of the Peetre--Fefferman--Stein maximal function 
$u^*(x)=\sup_{y\in\Rn}{|u(x-y)|}{(1+R|y|)^{-N}}$.
This was introduced in the theory of $F^s_{p,q}$ spaces in 1975 by
Peetre \cite{Pee75TL}, and soon adopted in the works of Triebel
\cite{T0,T2,T3} and others. The systematic use of $u^*$ for control of
pseudo-differential operators, cf.\ \eqref{apest-eq}, was seemingly
first proposed in \cite{JJ11pe}.

The symbol factor $F_a(x)$ in \eqref{apest-eq} is easily controlled in terms of
integrals reminiscent of the Mihlin--H{\"o}rmander multiplier theorem;
cf.\ Theorem~\ref{Fa-thm} below.
This is useful for type $1,1$-operators, because the integrals themselves
can be controlled for symbols in the self-adjoint subclass via their 
characterisation of H{\"o}rmander recalled in
Theorem~\ref{a*-thm} below.
In fact, in Section~\ref{Lp-sect} this has lead to estimates of such operators 
in spaces with $0<p\le 1$, which cannot be treated as duals of other spaces.

Notation is settled in Section~\ref{prel-sect} along with facts on
operators of type $1,1$. Section~\ref{pe-sect} 
briefly recalls some facts on \eqref{apest-eq} from \cite{JJ11pe}.
Littlewood--Paley analysis of type $1,1$-operators is treated systematically
in Section~\ref{corona-sect}. Estimates in spaces over $L_p$ are discussed in 
Section~\ref{Lp-sect}.

\section{Preliminaries on type $1,1$-operators}   \label{prel-sect}

Notation and notions from distribution theory, such as the spaces
$C^\infty_0$, $\cal S$, $C^\infty$ of smooth functions and
their duals $\cal D'$, $\cal S'$, $\cal E'$ of distributions,
and the Fourier transformation $\cal F$, will be as in H{\"o}rmander's
book \cite{H}, unless otherwise is mentioned.
Eg $\dual{u}{\varphi}$ denotes the value of a distribution $u$ on a test
function $\varphi$. 
The space $\cal O_M(\Rn)$ consists of the slowly increasing $f\in
C^\infty(\Rn)$, ie the $f$ that for each multiindex $\alpha$ and
some $N>0$ fulfils 
$|D^\alpha f(x)|\le c(1+|x|)^{N}$.

As usual $t_{+}=\max(0,t)$ is the positive part
and $[t]$ denotes the greatest integer $\le t$. In general, $c$ will denote
a real constant specific to the place of occurrence.

\subsection{The general definition of type $1,1$-operators}
The reader may consult \cite{JJ08vfm} for an overview of results on 
type $1,1$-operators and a systematic treatment. 
The present paper is partly a continuation of
\cite{JJ04DCR,JJ05DTL,JJ08vfm}, but it suffices to recall a few facts. 

The operators are defined, as usual, on the Schwartz space $\cal
S(\Rn)$ by
\begin{equation}
  a(x,D)u=\OP(a)u(x)
  =(2\pi)^{-n}\int e^{\im x\cdot \eta} a(x,\eta)\cal Fu(\eta)\,d\eta,
 \qquad u\in \cal S(\Rn).
  \label{axDu-id}
\end{equation}
Hereby the symbol $a(x,\eta)$ is required to be in 
$C^\infty(\Rn\times \Rn)$, of order $d\in \R$ and type $1,1$; ie for all
multiindices $\alpha$, $\beta\in \N_0^n$ it fulfils \eqref{Cab-ineq},
or more precisely has finite seminorms
\begin{equation}
 p_{\alpha,\beta}(a):= \sup_{x,\eta\in \Rn}
  (1+|\eta|)^{-(d-|\alpha|+|\beta|)}
  |D^\alpha_\eta D^\beta_x a(x,\eta)|
  <\infty.  
  \label{pab-eq}
\end{equation}
The Fr\'echet space of such symbols is denoted by $S^d_{1,1}(\Rn\times \Rn)$,
or just $S^d_{1,1}$.

For arbitrary $u\in \cal S'\setminus \cal S$ it is quite delicate
whether or not $a(x,D)u$ is defined.
To recall from \cite{JJ08vfm} how type $1,1$-operators 
can be defined in general, note that in terms of the partially Fourier
transformed  symbol 
\begin{equation}
  \hat a(\xi,\eta)=\cal F_{x\to\xi}(a(x,\eta)),
\end{equation}
one can define a modified symbol $\psi(2^{-m}D_x)a(x,\eta)=
\cal F^{-1}_{\xi\to x}(\psi(2^{-m}\xi)\hat a(\xi,\eta))$.

\begin{defn}   \label{a11-defn}
For a symbol $a(x,\eta)$ in $S^d_{1,1}(\Rn\times \Rn)$ and
cut-off functions $\psi\in C^\infty_0(\Rn)$ equal to
$1$ in a neighbourhood of the origin, let
\begin{equation}
  a_{\psi}(x,D)u:=
  \lim_{m\to\infty }\op{OP}(\psi(2^{-m}D_x)a(x,\eta)\psi(2^{-m}\eta))u.
  \label{aPsi-eq}
\end{equation}
If for each such $\psi$ the limit 
$a_{\psi}(x,D)u$ exists in $\cal D'(\Rn)$ and moreover is independent of $\psi$, then
$u$ belongs to the domain $D(a(x,D))$ by definition and 
\begin{equation}
  a(x,D)u=a_{\psi}(x,D)u.
  \label{aPsi'-eq}
\end{equation}
Thus $a(x,D)$ is a map $\cal S'(\Rn)\to\cal D'(\Rn)$ with dense domain.
\end{defn}

Obviously the action on $u$ is well defined for each $m$ in
\eqref{aPsi-eq} as the modified symbol is in $S^{-\infty }$.
Since the removal of high frequencies in $x$ and $\eta$, which is achieved from
$\psi(2^{-m}D_x)$ and $\psi(2^{-m}\eta)$, disappears for $m\to\infty $, this
was called definition by vanishing frequency modulation in \cite{JJ08vfm};
and accordingly $\psi$ is said to be a modulation function.

While the calculus of type $1,1$-operators is delicate in general,
cf \cite{H88,H89,H97}, the following result is straightforward from the
definition:

\begin{prop}
  \label{abc-prop}
When $a(x,\eta)$ is in $ S^{d_1}_{1,1}(\Rn\times \Rn)$ and
$b(\eta)$ belongs to $S^{d_2}_{1,0}(\Rn\times \Rn)$, 
then the symbol $c(x,\eta):=a(x,\eta)b(\eta)$
is in $S^{d_1+d_2}_{1,1}(\Rn\times \Rn)$ and
\begin{equation}
  c(x,D)u=a(x,D)b(D)u,
\end{equation}
where $D(c(x,D))=D(a(x,D)b(D))$;
that is, the two sides are simultaneously defined.
\end{prop}
\begin{proof}
That $c(x,\eta)$ is in $S^{d_1+d_2}_{1,1}$ can be verified in the usual way
from symbolic estimates.
For an arbitrary modulation function $\psi$ it is obvious from
\eqref{axDu-id} that for every $u\in \cal S$,
\begin{equation}
  \OP(\psi(2^{-m}D_x)a(x,\eta)\psi(2^{-m}\eta))b(D)u=
  \OP(\psi(2^{-m}D_x)a(x,\eta)\psi(2^{-m}\eta)b(\eta))u.
\end{equation}
This extends to all $u\in \cal S'$ since the symbols are in $S^{-\infty }$
or $S^{d_2}_{1,0}$.
Moreover, for $m\to\infty $ the limit exists on both or none of the two
sides for each $u\in \cal S'$, so in the notation of
\eqref{aPsi-eq}, 
\begin{equation}
  a_\psi(x,D)(b(D)u)= c_{\psi}(x,D)u. 
\end{equation}
Now $u\in D(c(x,D))$ if and only if the right-hand side is independent of
$\psi$, ie if the left-hand side is so, which is equivalent to 
$b(D)u\in D(a(x,D))$, ie to $u\in D(a(x,D)b(D))$.
\end{proof}

\begin{exmp}
  \label{Ching-exmp}
A standard example of a symbol of type $1,1$ results by taking an auxiliary
function $A\in C^\infty_0(\Rn)$, say with $\supp A\subset\{\,\eta\mid
\tfrac{3}{4}\le |\eta|\le\tfrac{5}{4}\,\}$, and $\theta\in \Rn$ fixed:
\begin{equation}
  a_{\theta}(x,\eta)
=\sum_{j=0}^\infty 2^{jd}e^{-\im 2^j{x}\cdot{\theta}}
  A(2^{-j}\eta).
  \label{ching-eq}
\end{equation}
Clearly $a_{\theta}\in S^d_{1,1}$ since the terms are disjointly
supported.

Such symbols were used by Ching \cite{Chi72} and 
Bourdaud \cite{Bou88} for $d=0$, $|\theta|=1$ to show unboundedness on $L_2$.
Refining this, H{\"o}rmander~\cite{H88}
linked continuity from $H^s$ with $s>-r$
to the property that $\theta$ is a zero of $A$ of order $r\in\N_0$.
Extension to $d\in \R$ was given in \cite{JJ08vfm}.

Moreover, it was shown in \cite[Lem.~3.2]{JJ08vfm} that
$a_{\theta}(x,D)$ is unclosable in $\cal S'$ when $A$ is taken to have
support in a small neighbourhood of $\theta$.
Therefore Definition~\ref{a11-defn} cannot in general be replaced 
by a closure of the graph in eg $\cal S'\times \cal S'$.
\end{exmp}

As a general result, it was shown in \cite[Sec.~4]{JJ08vfm}
that the subspace $\cal S(\Rn)+\cal F^{-1}\cal E'(\Rn)$ always is contained in
the domain of $a(x,D)$ and that this is a map
\begin{equation}
  a(x,D)\colon \cal S(\Rn)+\cal F^{-1}\cal E'(\Rn) \to \cal O_M(\Rn).
  \label{aFE-eq}
\end{equation}
In fact, if
$u=v+v'$ is an arbitrary splitting of $u$ 
with $v\in \cal S$ and $v'\in \cal F^{-1}\cal E'$, it was shown that
\begin{equation}
  a(x,D)u= a(x,D)v+\OP(a(1\otimes \chi))v',
  \label{aFE-id}
\end{equation}
whereby $a(1\otimes \chi)(x,\eta)=a(x,\eta)\chi(\eta)$ and $\chi\in
C^\infty_0(\Rn)$ is chosen so that $\chi=1$ holds in a neighbourhood of 
$\supp\cal Fv'$, but otherwise arbitrarily.
Here $a(x,\eta)\chi(\eta)$ is in $S^{-\infty }=\bigcap S^d_{1,1}$.

In fact, $\cal O_M(\Rn)$ is invariant under $a(x,D)$, and 
$a(x,D)\colon C^\infty\bigcap\cal S'\to C^\infty$; cf \cite[Thm.~2.7]{JJ10tmp}.

\subsection{Conditions along the twisted diagonal}   \label{TDC-ssect}
As the first explicit condition on the symbol of a type $1,1$-operator,
H{\"o}rmander \cite{H88} proved that \eqref{s-ineq} holds for 
the norms of $H^s$ with arbitrary $s\in\R$, $u\in\cal S$,
whenever $a\in S^d_{1,1}$
fulfils the \emph{twisted diagonal condition}: for some $B\ge 1$ 
\begin{equation}
  \hat a(\xi,\eta)=0 \quad\text{where}\quad
    B(1+|\xi+\eta|)< |\eta|.
  \label{tdc-cnd}
\end{equation}
This means that the partially Fourier transformed symbol 
$\hat a(\xi,\eta)$ vanishes in a conical neighbourhood of a
non-compact part of the twisted diagonal 
\begin{equation}
  \cal T=\{\,(\xi,\eta)\in \Rn\times \Rn\mid \xi+\eta=0\,\}.  
\end{equation}

Localisations to conical neighbourhoods (of non-compact parts) of $\cal T$
was also introduced by H{\"o}rmander in \cite{H88,H89,H97},
by passing to $a_{\chi,\varepsilon}(x,\eta)$ defined by
\begin{equation}
  \hat a_{\chi,\varepsilon}(\xi,\eta)
  =\hat a(\xi,\eta)\chi(\xi+\eta,\varepsilon\eta),
  \label{axe-eq}
\end{equation}
whereby $\chi\in C^\infty (\Rn\times \Rn)$ is chosen so that
$\chi(t\xi,t\eta)= \chi(\xi,\eta)$ for $t\ge 1,\ |\eta|\ge 2$ and
\begin{gather}
  \supp \chi\subset \{\,(\xi,\eta)\mid 1\le |\eta|,\ |\xi|\le
|\eta|\,\}
\label{chi1-eq} \\
  \chi=1 \quad\text{in}\quad 
  \{\,(\xi,\eta)\mid 2\le |\eta|,\ 2|\xi|\le |\eta|\,\}.
  \label{chi3-eq}
\end{gather}
Using this, H{\"o}rmander analysed a
milder condition than the strict vanishing in \eqref{tdc-cnd},
namely that for some $\sigma\in\R$, 
it holds for all multiindices $\alpha$ and $0<\varepsilon<1$ that
\begin{equation}
  N_{\chi,\varepsilon,\alpha}(a):=
  \sup_{R>0,\; x\in \Rn}R^{-d}\big(
  \int_{R\le |\eta|\le 2R} |R^{|\alpha|}D^\alpha_{\eta}a_{\chi,\varepsilon}
  (x,\eta)|^2\,\frac{d\eta}{R^n}
  \big)^{1/2}
  \le c_{\alpha,\sigma} \varepsilon^{\sigma+n/2-|\alpha|}.
  \label{Hsigma-eq}
\end{equation}
This asymptotics for $\varepsilon\to0$ always holds for $\sigma=0$,
as was proved in \cite[Lem.~9.3.2]{H97}.

For $\sigma>0$ the faster convergence to $0$ in \eqref{Hsigma-eq} 
was proved in \cite{H89} to imply that $a(x,D)$ is bounded on $u\in
\cal S(\Rn)$,
\begin{equation}
  \nrm{a(x,D)u}{H^{s}}\le c_s \nrm{u}{H^{s+d}} \quad\text{for}\quad s>-\sigma.
  \label{Hssigma-eq}
\end{equation}
The reader could consult \cite[Thm.~9.3.5]{H97} for this (and
\cite[Thm.~9.3.7]{H97} for four pages of proof of necessity of 
$s\ge -\sup\sigma$, with supremum over all $\sigma$ for which
\eqref{Hsigma-eq} holds).

If $\hat a$ is so small along $\cal T$ that 
\eqref{Hsigma-eq} holds for all $\sigma\in \R$, 
consequently there is boundedness $H^{s+d}\to H^s$ for
all $s\in \R$. Eg this validity of \eqref{Hsigma-eq} for all $\sigma$ is implied
by \eqref{tdc-cnd}, for since
\begin{equation}
  \supp \hat a_{\chi,\varepsilon}\subset 
  \{\,(\xi,\eta)\mid 1+|\xi+\eta|\le 2\varepsilon|\eta| \,\},
\label{a*tdc-eq}
\end{equation}
it is clear that \eqref{tdc-cnd} gives 
$a_{\chi,\varepsilon}\equiv 0$ whenever $0<2\varepsilon<1/B$.

More generally \eqref{Hsigma-eq} enters a characterisation of
the $a\in S^d_{1,1}$ for which the adjoint symbol
\begin{equation}
  a^*(x,\eta)=e^{\im D_x\cdot D_\eta}\bar a(x,\eta)  
\end{equation}
is again in $S^d_{1,1}$; cf the below condition \eqref{a*-cnd}. 
Since adjoining is an involution, such symbols constitute the class
\begin{equation}
  \tilde S^d_{1,1}:= S^d_{1,1}\cap (S^d_{1,1})^* .
\end{equation}

\begin{thm}   \label{a*-thm}
For a symbol $a(x,\eta)$ in $S^d_{1,1}(\Rn\times \Rn)$ the following
properties are equivalent:
\begin{rmlist}
    \item   \label{a*-cnd}
    $a(x,\eta)$ belongs to $\tilde S^d_{1,1}(\Rn\times \Rn)$.
  \item   \label{orderN-cnd}
  For arbitrary $N>0$ and $\alpha$, $\beta$ there is a constant
$C_{\alpha,\beta,N}$ such that 
\begin{equation}
  |D^\alpha_\eta D^\beta_x a_{\chi,\varepsilon}(x,\eta)|\le C_{\alpha,\beta,N}
  \varepsilon^{N}(1+|\eta|)^{d-|\alpha|+|\beta|}
  \quad\text{for}\quad 0<\varepsilon<1.
\end{equation} 
  \item   \label{sigma-cnd}
  The seminorm $N_{\chi,\varepsilon,\alpha}(a)$ fulfils
  \eqref{Hsigma-eq} for all $\sigma\in\R$. 
\end{rmlist}
In the affirmative case $a\in \tilde S^{d}_{1,1}$, and there is an estimate
\begin{equation}
  |D^\alpha_\eta D^\beta_x a^*(x,\eta)| \le 
  (C_{\alpha,\beta}(a)+C'_{\alpha,\beta,N})(1+|\eta|)^{d-|\alpha|+|\beta|}
\end{equation}
for a certain continuous seminorm 
$C_{\alpha,\beta}$ on $S^d_{1,1}(\Rn\times \Rn)$ and a finite sum
$C'_{\alpha,\beta,N}$ of constants fulfilling 
the inequalities in \eqref{orderN-cnd}.
\end{thm}

It should be observed that $a(x,\eta)$ fulfils
\eqref{a*-cnd}  if and only if 
$a^*(x,\eta)$ does so (neither \eqref{orderN-cnd} nor
\eqref{sigma-cnd} make this obvious).
But \eqref{orderN-cnd} immediately gives the inclusion 
$\tilde S^d_{1,1}\subset \tilde S^{d'}_{1,1}$  for $d'>d$.
Condition \eqref{sigma-cnd} is
close in spirit to the Mihlin--H{\"o}rmander multiplier theorem
and is useful for the estimates to follow in Section~\ref{Lp-sect}.

The theorem was undoubtedly known to 
H{\"o}rmander, who stated the equivalence of \eqref{a*-cnd} 
and \eqref{orderN-cnd} explicitly in \cite[Thm.~4.2]{H88} and
\cite[Thm.~9.4.2]{H97}, in the latter 
with brief remarks on \eqref{sigma-cnd}. 

As a corollary to the proof of Theorem~\ref{a*-thm}, for which the
reader also may consult \cite{JJ10tmp}, the vanishing
frequency modulation gave the following main result in \cite[Thm.~4.6]{JJ10tmp}:

\begin{thm}   \label{tildeS-thm}
If $a(x,\eta)$ is in the class $\tilde
S^d_{1,1}(\Rn\times \Rn)$, characterised in Theorem~\ref{a*-thm}, then
\begin{equation}
  a(x,D)\colon \cal S'(\Rn)\to \cal S'(\Rn)
\end{equation}
is everywhere defined and continuous, and it equals the adjoint of 
$\OP(e^{\im D_x\cdot D_\eta}\bar a(x,\eta))$.
\end{thm}

\section{Pointwise estimates}   \label{pe-sect}
A main technique in this paper will be to estimate 
$|a(x,D)u(x)|$ at an arbitrary point of $\Rn$.
The recent results on this by the author \cite{JJ11pe} are recalled here for
convenience of the reader.

\subsection{The factorisation inequality}   \label{revw-ssect}
When $\supp \hat u$ is compact in $\Rn$, the action on $u$ by
$a(x,D)$ can be \emph{separated} from $u$ at the cost of an estimate,
which is the \emph{factorisation inequality}
\begin{equation}
  |a(x,D)u(x)|\le F_a(N,R;x) u^*(N,R;x).
  \label{Fau*-eq}
\end{equation} 
Here $u^*$ denotes the maximal function of Peetre--Fefferman--Stein type,
defined as
\begin{equation}
  u^*(N,R;x)=\sup_{y\in \Rn}\frac{|u(x-y)|}{(1+R|y|)^N}
  =\sup_{y\in \Rn}\frac{|u(y)|}{(1+R|x-y|)^N}
  \label{u*-eq}
\end{equation}
when $\supp\hat u\subset \overline{B}(0,R)$; cf.\ \eqref{aFE-id}.
The parameter $N$ may eg be chosen so that $N\ge \order(\hat u)$.

The symbol factor $F_a(x)$ only depends on $u$ in
a vague way, viz.\ through $N$
and $R$:
\begin{equation}
  F_a(N,R;x)= \int_{\Rn} (1+R|y|)^N 
  |\cal F^{-1}_{\eta\to y}(a(x,\eta)\chi(\eta ))|\,dy,
  \label{Fa-id}
\end{equation}
where the auxiliary function $\chi\in C^\infty_0(\Rn)$ 
should equal $1$ on a neighbourhood of
$\supp\hat u$.
However, $\chi$ is left out from the notation 
in $F_a(x)$, as this would be redundant by the results below in
Theorem~\ref{Fa-thm}.

The estimate \eqref{Fau*-eq} is useful as both factors are easily
controlled. Eg $u^*(x)$ is polynomially bounded, for
$|u(y)|\le c(1+|y|)^N\le c(1+R|y-x|)^N(1+|x|)^N$ holds according
to the Paley--Wiener--Schwartz Theorem if $N\ge \order(\hat u)$, $R\ge 1$,
and by \eqref{u*-eq} this implies
\begin{equation}
  u^*(N,R;x)\le c (1+|x|)^N,\qquad x\in \Rn.
  \label{u*PWS-eq}
\end{equation}

The non-linear map $u\mapsto u^*$ is also bounded with
respect to the $L_p$-norm on the subspace $L_p\cap\cal F^{-1}\cal
E'$. This can be shown in an elementary way; cf \cite[Thm.~2.6]{JJ11pe}.

Secondly, for the symbol factor one has 
$F_a\in C(\Rn)\cap L_\infty (\Rn)$ with estimates
highly reminiscent of the Mihlin--H{\"o}rmander conditions for Fourier
multipliers:

\begin{thm}
  \label{Fa-thm}
Assume the symbol
$a(x,\eta)$ is in $S^d_{1,1}(\Rn\times\Rn)$ and let $F_a(N,R;x)$ be given by  
\eqref{Fa-id} for parameters $R,N>0$, with the auxiliary function
taken as $\chi=\psi(R^{-1}\cdot)$ for $\psi\in C^\infty_0(\Rn)$ equalling
$1$ in a set with non-empty interior. Then it holds for all $x\in \Rn$ that
\begin{equation}
 0\le F_a(x) \le c_{n,N} \sum_{|\alpha|\le [N+\frac{n}{2}]+1} 
  \Big(\int_{R\supp \psi} |R^{|\alpha|}D^{\alpha}_\eta a(x,\eta)|^2
    \,\frac{d\eta}{R^n}
  \Big)^{1/2}.
  \label{FaMH-eq}
\end{equation}
\end{thm}

For the elementary proof the reader is referred to
Theorem~4.1 and Section~6 in \cite{JJ11pe}.

\begin{rem}
  \label{Fa-rem}
A further analysis of $F_a$'s dependence on $R$ was given in
\cite[Sect.~4]{JJ11pe}. Eg when the cut-off function $\psi$ in
Theorem~\ref{Fa-thm} vanishes around the origin, then 
$F_a(x)=\cal O(R^d)$ for $a\in S^d_{1,1}$.
Moreover, when this is applied to symbols of the form
$a_Q(x,\eta)=\varphi(Q^{-1}D_x)a(x,\eta)$, $Q>0$, 
with $\varphi=0$ around the origin,
there is the sharpening $F_{a_Q}(x)=\cal O(Q^{-M}R^{d+M})$.
\end{rem}

\section{Littlewood--Paley analysis}
  \label{corona-sect}

For type $1,1$-operators, Littlewood--Paley analysis will most
conveniently depart from the limit in \eqref{aPsi-eq} with an arbitrary modulation
function $\psi$.
As $\psi$ is a test function, this gives in the usual way a
Littlewood--Paley decomposition $1=\psi(\eta)+\sum_{j=1}^\infty 
\varphi(2^{-j}\eta)$ by setting $\varphi=\psi-\psi(2\cdot )$.
Note here that if $\psi\equiv 1$ for $|\eta|\le r$ while $\psi\equiv 0$ for
$|\eta|\ge R$, one can fix an integer $h\ge 2$ so that $2R< r2^h$.
Then
\begin{equation}
  \varphi(2^{-j}\eta)\ne0\implies 
  r2^{j-1}\le|\eta|\le R2^j.
\end{equation}

Inserting twice into \eqref{aPsi-eq} that
$\psi(2^{-m}\eta)=\psi(\eta)+\varphi(2^{-1}\eta)+
\dots +\varphi(2^{-m}\eta)$, the paradifferential splitting from the 1980's is
recovered: 
if $a(x,\eta)$ is in $S^d_{1,1}$, and
$u\in {\cal S}'(\Rn)$, then
\begin{equation}
  a_{\psi}(x,D)u=
  a_{\psi}^{(1)}(x,D)u+a_{\psi}^{(2)}(x,D)u+a_{\psi}^{(3)}(x,D)u,
  \label{a123-eq}
\end{equation}
whenever the three series below all converge in ${{\cal D}}'$ 
(cf.\ Remark~\ref{conv-rem}),
\begin{align}
    a_{\psi}^{(1)}(x,D)u&=\sum_{k=h}^\infty \sum_{j\le k-h} a_j(x,D)u_k
  =\sum_{k=h}^\infty a^{k-h}(x,D)u_k
  \label{a1-eq}\\
  a_{\psi}^{(2)}(x,D)u&= \sum_{k=0}^\infty
               \bigl(a_{k-h+1}(x,D)u_k+\dots+a_{k-1}(x,D)u_k+a_{k}(x,D)u_k
\notag\\[-2\jot]
   &\hphantom{= \sum_{k=0}^\infty\bigl(a_{k-h+1}(x,D)u_k+\dots{}}%
                {}+a_{k}(x,D)u_{k-1} +\dots+a_k(x,D)u_{k-h+1}\bigr) 
  \label{a2-eq}\\
   a_{\psi}^{(3)}(x,D)u&=\sum_{j=h}^\infty\sum_{k\le j-h}a_j(x,D)u_k
   =\sum_{j=h}^\infty a_j(x,D)u^{j-h}.
  \label{a3-eq}
\end{align}
Here $u_k=\varphi(2^{-k}D)u$ while
$a_k(x,\eta)=\varphi(2^{-k}D_x)a(x,\eta)$;
by convention $\varphi$ is replaced by $\psi$ for $k=0$ and 
$u_k\equiv 0\equiv a_k$ for $k<0$.
In addition superscripts are used for the convenient shorthands 
$u^{k-h}=\psi(2^{h-k}D)u$
and $a^{k-h}(x,D)=\sum_{j\le k-h}a_j(x,D)=\op{OP}(\psi(2^{h-k}D_x)a(x,\eta))$. 
Using this, there is a
brief version of \eqref{a2-eq},
\begin{equation}
 a_{\psi}^{(2)}(x,D)u=\sum_{k=0}^\infty
((a^{k}-a^{k-h})(x,D)u_k+a_k(x,D)(u^{k-1}-u^{k-h})). 
  \label{a2'-eq}
\end{equation}
Occasionally the subscripts $\psi$ are omitted, as done already in the
summands in \eqref{a1-eq}--\eqref{a3-eq}.

The main point here is that the series have
the following inclusions for the spectra of the summands in 
\eqref{a1-eq}, \eqref{a3-eq} and \eqref{a2'-eq},
with $R_h=\tfrac{r}{2}-R2^{-h}>0$:
\begin{gather}
  \supp{\cal F}(a^{k-h}(x,D)u_k)\subset
  \bigl\{\,\xi\bigm| 
  R_h2^k\le|\xi|\le \tfrac{5R}{4} 2^k\,\bigr\},
  \label{supp1-eq}  \\
  \supp{\cal F}(a_k(x,D)u^{k-h})\subset
  \bigl\{\,\xi \bigm| 
  R_h2^k\le|\xi|\le \tfrac{5R}{4} 2^k\,\bigr\},
  \label{supp3-eq}  \\
  \supp{\cal F}\big(a_k(x,D)(u^{k-1}-u^{k-h})\big) \bigcup
  \supp{\cal F}\big((a^k-a^{k-h})(x,D)u_k\big)\subset
  \overline{B}(0,2R2^k)
  \label{supp2-eq}
\end{gather}
Such spectral corona and ball properties
have been known since the 1980's (e.g.\ \cite[(5.3)]{Y1}) 
although they were verified then only for
elementary symbols $a(x,\eta)$, in the sense of Coifman and
Meyer~\cite{CoMe78}. However, this  restriction is redundant because of 
the \emph{Spectral Support Rule}, which for 
$u\in {\cal F}^{-1}{\cal E}'(\Rn)$ states that
\begin{equation}
     \supp{\cal F}(a(x,D)u)\subset
\bigl\{\,\xi+\eta \bigm| (\xi,\eta)\in \supp{\cal F}_{x\to\xi\,} a,\ 
     \eta\in \supp{\cal F} u \,\bigr\},
  \label{Sigma-eq}
\end{equation}
A short proof of this can  be found in \cite[App.~B]{JJ10tmp}
(cf.\ also \cite{JJ05DTL,JJ08vfm} for the full version).
Since \eqref{supp1-eq}--\eqref{supp2-eq} follow easily from \eqref{Sigma-eq},
cf.\ \cite{JJ05DTL,JJ10tmp}, details are omitted. 

Recently the pointwise estimates in Remark~\ref{Fa-rem}
were utilised for the following result. It was deduced  in
\cite[Thm.~5.1]{JJ11pe}, with extension to type $1,1$ in Section~6
there. 

\begin{thm} \label{a123-thm}
  For each $a(x,\eta)$ in $S^d_{1,1}$
the decomposition \eqref{a123-eq} is valid with the terms in
\eqref{a1-eq}--\eqref{a3-eq}
having spectral relations \eqref{supp1-eq},\eqref{supp3-eq}, \eqref{supp2-eq}
and pointwise estimates, cf Section~\ref{pe-sect}, 
\begin{align}
  |a^{k-h}(x,D)u_k(x)|&\le p(a) (R2^{k})^d u_k^*(N,R2^k;x),
  \label{a1-pe} \\
  |(a^k-a^{k-h})(x,D)u_k(x)|& \le  p(a) (R2^{k})^d u_k^*(N,R2^k;x),
  \label{a2-pe} \\
  |a^k(x,D)(u^{k-1}(x)-u^{k-h}(x))|& \le p(a) (R2^{k})^d 
  \sum_{l=1}^{h-1} 2^{-ld}u_{k-l}^*(N,R2^{k-l};x),
  \label{a2'-pe} \\
  |a_j(x,D)u^{j-h}(x)|&\le c_M 2^{-jM} p(a) 
  \sum_{k=0}^j (R2^{k})^{d+M} u_k^*(N,R2^k;x).
  \label{a3-pe}
\end{align}
Hereby $p(a)$ denotes a continuous seminorm on $S^d_{1,1}$ and $M\in \N$.
\end{thm}

It is well known that in \eqref{a3-pe} one may treat 
the sum over $k$ by the elementary inequality 
\begin{equation}
  \sum_{j=0}^\infty 2^{sjq}
   (\sum_{k=0}^j |b_k|)^q\le c \sum_{j=0}^\infty 2^{sjq}|b_j|^q,   
\label{Yama-ineq}
\end{equation}
valid for all 
$b_j\in \C$ and $0<q\le \infty $ provided $s<0$; cf \cite{Y1}.

\begin{rem}
  \label{conv-rem}  
There is the addendum that the series 
\eqref{a1-eq}, \eqref{a3-eq} always converge for $u\in\cal S'$; 
so that $u$ is in $D(a(x,D))$
if and only if the $a^{(2)}$-series converges. Cf.\ \cite[Thm.~6.3]{JJ10tmp}.
\end{rem}

\begin{rem} \label{atdc-rem}
  If $a(x,\eta)$ satisfies the twisted diagonal condition
\eqref{tdc-cnd} for some $B\ge 1$,  
the supports in \eqref{supp2-eq} are for large $k$ both contained in the corona
\begin{equation}
  \bigl\{\,\xi \bigm| \frac{r}{2^{h+1}B} 2^k \le |\xi|\le 2R2^k\,\bigr\}.
  \label{supp2'-eq}
\end{equation}
Indeed, \eqref{Sigma-eq} yields that
$\supp\cal Fa_k(x,D)(u^{k-1}-u^{k-h})$ is contained in
\begin{equation}
  \bigl\{\,\xi+\eta \bigm| (\xi,\eta)\in \supp(\varphi_k\otimes 1)
\hat a,\ r2^{k-h}\le |\eta|\le R2^{k-1}\,\bigr\}.
\end{equation}
Therefore any $\zeta=\xi+\eta$ in the support fulfils $|\zeta|\le
R2^k+R2^{k-1}=(3R/2)2^k$.
But \eqref{tdc-cnd} implies that $B(1+|\xi+\eta|)\ge |\eta|$ on
$\supp\cal F_{x\to\xi}a$ so that, for all $k\ge h+1+\log_2(B/r)$,
\begin{equation}
  |\zeta|\ge \tfrac{1}{B}|\eta|-1\ge \tfrac{1}{B}r2^{k-h}-1\ge 
  (\tfrac{r}{2^hB}-2^{-k})2^k\ge \tfrac{r}{2^{h+1}B}2^k.
\end{equation}
The term $(a^k-a^{k-h})(x,D)u_k$ is analogous, but causes
$3R/2$ to be replaced by $2R$.
\end{rem}

\section{$L_p$-estimates}
  \label{Lp-sect}
\subsection{Function spaces}
To proceed from $H^s$-results, it would of course be natural to consider 
Sobolev spaces $H^s_p$ and H{\"o}lder--Zygmund spaces $C^s_*$
(cf.\ \cite[Def.~8.6.4]{H97}), but these are
special cases of the Besov spaces $B^{s}_{p,q}$ and Lizorkin--Triebel
spaces $F^{s}_{p,q}$. In fact,
\begin{gather}
 H^s_p=F^s_{p,2} \quad\text{for}\quad 1<p<\infty,   \label{HF-eq}
   \\
 C_*^s=B^s_{\infty ,\infty }\quad\text{for}\quad s\in\R.   \label{CB-eq}
\end{gather}
Because of the Littlewood--Paley analysis that will follow, it
requires almost no extra effort in the estimates to
cover the full $B^s_{p,q}$ and $F^s_{p,q}$ scales.

To invoke the $B^{s}_{p,q}$ and $F^{s}_{p,q}$ scales is natural in the
context, for it was shown in \cite{JJ04DCR,JJ05DTL} that every type
$1,1$-operator $a(x,D)$ of order $d\in\R$ is a bounded map 
\begin{equation}
 a(x,D)\colon F^d_{p,1}(\Rn)\to L_p(\Rn) \quad\text{for $1\le p<\infty$}.   
  \label{Fdp1-eq}
\end{equation}
Because $B^d_{p,1}\subset F^d_{p,1}$ is a strict inclusion for $p>1$, 
this sharpened the borderline analysis of Bourdaud
\cite{Bou88}; \eqref{Fdp1-eq} was moreover proved to be optimal 
within the $B^{s}_{p,q}$- and $F^{s}_{p,q}$-scales.

To recall the definition of $B^{s}_{p,q}$ and $F^{s}_{p,q}$, 
let a Littlewood--Paley partition of unity $1=\sum_{j=0}^\infty
\Phi_j$ be chosen as in Section~\ref{corona-sect} 
with $\Phi_j=\Phi(2^{-j}\cdot  )$ for
$\Phi=\Psi-\Psi(2\cdot )$, though $\Phi_0=\Psi$, 
whereby $\Psi\in C^\infty_0(\Rn)$ equal to $1$ around the origin is fixed. 
Usually it has been required that $\supp\Phi$ should be contained in the 
corona with $\tfrac{1}{2}\le |\xi|\le 2$; 
but this restriction is avoided here in order that $\Psi$ can be taken
equal to an arbitrary modulation function entering $a(x,D)$.
That this is possible can be seen by adopting the approach in eg
\cite{Y1,JoSi08}:

When $\Psi$ is fixed as above, then 
the spaces are defined for $s\in \R$ and $p,q\in
\,]0,\infty ]$ as follows, when
$\nrm{\cdot }{p}$ denotes the
(quasi-)norm of the Lebesgue space 
$L_p(\Rn)$ for $0<p\le \infty $ and $\nrm{\cdot }{\ell_q}$
stands for that of the sequence space $\ell_q(\N_0)$,
\begin{align}
 B^{s}_{p,q}(\Rn)&=\bigl\{\,u\in\cal S'(\Rn)\bigm|
 \Nrm{ \{2^{sj} \nrm{\Phi_j(D)u(\cdot)}{p} \}_{j=0}
 ^\infty}{\ell_q} <\infty\,\bigr\},
 \label{bspq-eq} 
    \\
 F^{s}_{p,q}(\Rn)&=\bigl\{\,u\in\cal S'(\Rn)\bigm|
 \Nrm{ \nrm{ \{2^{sj}\Phi_j(D) u\}_{j=0}
  ^\infty}{\ell_q} (\cdot)}{p} <\infty\,\bigr\}\,.
 \label{fspq-eq}
\end{align}
Throughout it will be understood that $p<\infty$ when
Lizorkin--Triebel spaces $F^{s}_{p,q}$ are considered. 

In the definition the finite expressions are norms for $p,q\ge 1$
(quasi-norms if $p<1$ or $q<1$).
In general $u\mapsto \nrm{u}{}^{\lambda}$ is subadditive for 
$\lambda\le\min(1,p,q)$, so $\nrm{f-g}{}^\lambda$ is a metric.

This implies continuous embeddings 
$\cal S\hookrightarrow B^{s}_{p,q}\hookrightarrow \cal S'$
and 
$\cal S\hookrightarrow F^{s}_{p,q}\hookrightarrow \cal S'$ in the usual way,
thence completeness (cf \cite{JoSi07,T2}). There are simple embeddings 
$F^{s}_{p,q}\hookrightarrow F^{s'}_{p,r}$ for $s'<s$ and arbitrary $q$,$r$,
or for $s'=s$ when $r\ge q$. Similarly for $B^{s}_{p,q}$.

\begin{exmp}
In the $F^{s}_{p,q}$-scale,  $f(t)=\sum_{j=0}^\infty 2^{-jd}e^{\im 2^jt}$
belongs locally to
$F^d_{p,\infty }(\R)$; cf \cite[Rem.~3.7]{JJ08vfm}.
This is for $0<d\le 1$ a variant of
Weierstrass' nowhere differentiable function.

Homogeneous distributions were characterised in the $B^{s}_{p,q}$-scale 
in Prop.~2.8 of \cite{JJ08par}:
when $u\in \cal D'(\Rn)$ is $C^\infty$ on $\Rn\setminus\{0\}$ and homogeneous
of degree $a\in \C$ there (cf \cite[Def~3.2.2]{H}),
then (at $x=0$) $u$ is locally  in 
$B^{\fracci np+\Re a}_{p,\infty}(\Rn)$ for $0<p\le\infty$. 
If $-n<\Re a<0$ and $p\in\,]-\tfrac{n}{\Re a},\infty]$ then $u\in B^{\fracci np+\Re
a}_{p,\infty}(\Rn)$; this holds also for $p=\infty$ if $\Re a=0$.
These conclusions are optimal for
$s$ and $q$, unless $u$ is a homogenenous polynomial 
(the only case in which $u\in  C^\infty(\Rn)$).
Eg $\delta_0\in B^{\fracci np}_{p,\infty} $ while a quotient of two homogeneous
polynomials of the same degree, say $P(x)/Q(x)$ is locally in
$B^{\fracci np}_{p,\infty }$ for $0<p\le \infty $.
\end{exmp}

Invoking a multiplier result, one finds a dyadic
ball and corona criterion:

\begin{lem}
  \label{Fspq-lem}
Let $s>\max(0,\fracc np-n)$ for $0<p<\infty$ and $0< q\le \infty$ and suppose
$u_j\in \cal S'(\Rn)$ fulfil that, for some $A>0$,
\begin{equation}
  \supp\cal F u_j\subset B(0,A2^j),\qquad
  F(q):=
  \Nrm{(\sum_{j=0}^\infty 2^{sjq}|u_j(\cdot)|^q)^{\fracci1q}}{p}<\infty.
  \label{ballFq-eq}
\end{equation}
Then $\sum_{j=0}^\infty u_j$ converges in $\cal S'(\Rn)$ to some $u\in
F^s_{p,r}(\Rn)$ for $ r\ge q$, $ r>\tfrac{n}{n+s}$,
and $\nrm{u}{F^s_{p,r}}\le cF(r)$ for some $c>0$ depending on
$n$, $s$, $p$ and $r$.

When moreover $\supp\cal Fu_j\subset \{\,\xi\mid A^{-1}2^j\le |\xi|\le
A2^j\,\}$ for all $j\ge J$, for some $J\ge 1$, 
then the conclusions are valid for all $s\in \R$ and $r=q$.
\end{lem}

This is an isotropic version of \cite[Lem.~3.19-20]{JoSi08}, where the proof 
is applicable for arbitrary Littlewood--Paley partitions, though with 
other constants if $\Psi$ is such
that $R>2$. Alternatively the reader may refer to the below
Proposition~\ref{Fcor-prop}, where the proof also covers the sufficiency of 
\eqref{ballFq-eq} and as a special case gives the last part of
Lemma~\ref{Fspq-lem} as well.

From Lemma~\ref{Fspq-lem} it follows that $F^{s}_{p,q}$ is 
independent of the particular Littlewood--Paley
decomposition, and that different choices lead to equivalent quasi-norms.

The functions $u_k=\Phi(2^{-k}D)u$ will play a central role below
because their maximal functions $u_k^*$, cf Section~\ref{pe-sect}, are
controlled in terms of the Lizorkin--Triebel norm $\nrm{u}{F^{s}_{p,q}}$ as follows: 
for $0<t<\infty $ there is an
estimate, cf \cite[Thm.~2.10]{Y1}, in terms of the modified Hardy-Littlewood
maximal function given by
$M_t u_k(x)=\sup_{r>0} (r^{-n}\int_{|x-y|\le r}|u(y)|^t\,dy)^{1/t}$,
\begin{equation}
  u_k^*(N,R2^k;x)\le   u_k^*(\tfrac{n}{t},R2^k;x)\le c M_tu_k(x),
 \quad  N\ge n/t.
  \label{Yam-ineq}
\end{equation}
So for $t<\min(p,q)$ the Fefferman-Stein inequality (cf
\cite[Thm.~2.2]{Y1}) yields a basic inequality  valid for the 
$u_k^*=u_k^*(N,R2^k,\cdot )$ and any $s\in \R$, 
\begin{equation}
  \int_{\Rn} \nrm{2^{sk}u_k^*(\cdot )}{\ell_{q}}^p\,dx
  \le 
  c\int_{\Rn} \nrm{2^{sk}M_t u_k(\cdot )}{\ell_{q}}^p\,dx
  \le  c'
  \int_{\Rn} \nrm{2^{sk}u_k(\cdot )}{\ell_{q}}^p\,dx
  =c'\nrm{ u}{F^{s}_{p,q}}^p.
  \label{u*M-eq}
\end{equation}

As general references to the theory of these function spaces, the reader is
referred to the books \cite{RuSi96,T2,T3}; the paper 
\cite{Y1} gives a concise (anisotropic) presentation.

\begin{rem}
  \label{Marschall-rem}
As an alternative to the techniques in Section~\ref{pe-sect}, there is 
an estimate for symbols $b(x,\eta)$ in $L_{1,\loc}(\R^{2n})\cap
\cal S'(\R^{2n})$ with support in $\Rn\times \overline{B}(0,2^k)$ and $\supp\cal F
u\subset \overline{B}(0,2^k)$, $k\in\N$:
\begin{equation}
  |b(x,D)v(x)|\le c\Nrm{b(x,2^k\cdot )}{\dot B^{n/t}_{1,t}} M_t u(x),
  \qquad 0<t\le 1.
  \label{Marschall-ineq}
\end{equation}
This is Marschall's inequality, it goes back to \cite[p.37]{Mar85} and 
was exploited in eg \cite{Mar91}; in the above form it was proved in
\cite{JJ05DTL} under the condition that the right-hand side is in
$L_{1,\loc}(\Rn)$ (cf also \cite{JoSi08}).
While $M_tu$ is as in \eqref{Yam-ineq}, the norm 
$\Nrm{b(x,2^k\cdot )}{\dot B^{n/t}_{1,t}}$
of the symbol in the homogenenous Besov space is of special interest here.
It is defined in terms of a partition of unity $1=\sum_{j=-\infty }^\infty
\Phi(2^{-j}\eta)$, with $\Phi$ as in \eqref{bspq-eq}, and \eqref{bspq-eq}
read with $\ell_q$ over $\Z$ gives the norm. This yields the
well-known dyadic scaling property that
\begin{equation}
  \Nrm{b(x,2^k\cdot )}{\dot B^{n/t}_{1,t}}
  = 2^{k(\frac nt-n)}\Nrm{b(x,\cdot)}{\dot B^{n/t}_{1,t}}.  
\end{equation}
\end{rem}

\subsection{Basic estimates in $L_p$}
For general type $1,1$-operators $a(x,D)$ one has the next result. 
This appeared 
in \cite[Cor.~6.2]{JJ05DTL}, albeit with a rather sketchy
explanation. Therefore a full proof is given here, now explicitly based on
Definition~\ref{a11-defn} and the pointwise techniques in Section~\ref{pe-sect}:

\begin{thm}
  \label{FBspq-thm}
Every $a(x,D)$ in $\OP(S^d_{1,1}(\Rn\times\Rn))$ is continuous,
for all $s>\max(0,\fracc np-n)$, $0<p,q\le\infty$, 
\begin{align}
   a(x,D)&\colon F^{s+d}_{p,q}(\Rn)\to F^s_{p,r}(\Rn),
   \quad p<\infty,\ r\ge q,\ r>n/(n+s),
  \label{Fspr-eq} \\
   a(x,D)&\colon B^{s+d}_{p,q}(\Rn)\to B^s_{p,q}(\Rn).
  \label{Bspq-eq}
\end{align}
Here the twisted diagonal condition \eqref{tdc-cnd} implies 
\eqref{Fspr-eq} and \eqref{Bspq-eq} for all $s\in \R$ and $r=q$.
\end{thm}
\begin{proof}
Let $\psi$ denote an arbitrary modulation function, and recall the notation
from Section~\ref{corona-sect}, in particular \eqref{a123-eq}
and $R$, $r$ and $h$. It is exploited below that $\nrm{u}{F^{s}_{p,q}}$ can
be calculated in terms of the Littlewood--Paley partition associated with
$\psi$. 

For $a^{(1)}(x,D)u=\sum_{k=h}^\infty a^{k-h}(x,D)u_k$ and $u\in F^{s}_{p,q}$,
application of the norms of $\ell_q$ and $L_p$ to the pointwise estimate in \eqref{a1-pe}
gives (if $q<\infty $ for simplicity's sake)
\begin{equation}
 \int_{\Rn}(\sum_{k=0}^\infty 2^{skq}
|a^{k-h}(x,D)u_k(x)|^q)^{\frac{p}{q}}\,dx 
\le 
  c_2p(a)^p\Nrm{(\sum_{k=0}^\infty 2^{(s+d)kq}u_k^*(x)^q)^\frac1q}{p}^p.
  \label{a1Lplq-eq}
\end{equation}
Taking $N>n/\min(p,q)$ in $u^*_k$, it is seen from \eqref{u*M-eq} that one has the bound in 
Lemma~\ref{Fspq-lem}  for all $s\in \R$, whilst the corona condition there holds
by Theorem~\ref{a123-thm}, so the lemma gives
\begin{equation}
  \nrm{a^{(1)}(x,D)u}{F^{s}_{p,q}}\le c 
  (\int_{\Rn}(\sum_{k=0}^\infty 2^{skq}
|a^{k-h}(x,D)u_k(x)|^q)^{\frac{p}{q}}\,dx)^{\fracpi}
   \le c'\nrm{u}{F^{s+d}_{p,q}}.
  \label{a1Fspq-eq}
\end{equation}

In the contribution $a^{(3)}(x,D)u=\sum_{j=h}^\infty a_j(x,D)u^{j-h}$
one may apply \eqref{Yama-ineq}. For $M>s$ this gives
\begin{equation}
  \begin{split}
 \sum_{j=0}^\infty 2^{sjq}   |a_j(x,D)u^{j-h}(x)|^q
&\le \sum_{j=0}^\infty 2^{(s-M)jq} (\sum_{k=0}^{j} 
      c_Mp(a)(R2^k)^{d+M} u_k^*(N,R2^k;x))^q
\\
&\le  cp(a)^q \sum_{j=0}^\infty 2^{(s+d)jq} u_j^*(N,R2^j;x)^q.
  \end{split}
\end{equation}
Proceeding by integration one arrives at
\begin{equation}
 (\int_{\Rn}(\sum_{j=0}^\infty 2^{sjq}
|a_j(x,D)u^{j-h}(x)|^q)^{\frac{p}{q}}\,dx )^{\fracpi}
\le 
  c_3p(a)\Nrm{(\sum_{j=0}^\infty 2^{(s+d)jq}u_j^*(x)^q)^\frac1q}{p}.
  \label{a3Lplq-eq}
\end{equation}
Hence the same application of Lemma~\ref{Fspq-lem} as for \eqref{a1Fspq-eq} 
now gives $\nrm{a^{(3)}(x,D)u}{F^{s}_{p,q}}\le c\nrm{u}{F^{s+d}_{p,q}}$.

In estimates of $a^{(2)}(x,D)u$ the terms can be treated similarly, now departing from 
\eqref{a2-pe} and \eqref{a2'-pe}. Thus one finds
\begin{multline}
  \big( \int_{\Rn}(\sum_{k=0}^\infty 2^{skq}
|(a^k-a^{k-h})(x,D)u_k(x)+a_k(x,D)(u^{k-1}-u^{k-h})|^q)^{\frac{p}{q}}\,dx 
\big)^{\fracpi}
  \\
\le 
  c'_2p(a)\Nrm{(\sum_{k=0}^\infty 2^{(s+d)kq}u_k^*(x)^q)^\frac1q}{p}.
  \label{a2Lplq-eq}
\end{multline}
In case \eqref{tdc-cnd} holds, Remark~\ref{atdc-rem}
shows that Lemma~\ref{Fspq-lem} is applicable once more, so the argument for \eqref{a1Fspq-eq} 
gives $\nrm{a^{(2)}(x,D)u}{F^{s}_{p,q}}\le c\nrm{u}{F^{s+d}_{p,q}}$.
So for all $s\in \R$,
\begin{equation}
  \nrm{a_{\psi}(x,D)u}{F^{s}_{p,q}}
\le  \sum_{j=1,2,3} \nrm{a^{(j)}(x,D)u}{F^{s}_{p,q}}
\le c p(a)\nrm{u}{F^{s+d}_{p,q}}.
  \label{apsi123-eq}
\end{equation}
Without \eqref{tdc-cnd} the spectra are by \eqref{supp2-eq} just
contained in balls, but the condition $s>\max(0,\fracnp-n)$ and those on 
$r$ imply that 
$\nrm{a^{(2)}(x,D)u}{F^{s}_{p,r}}\le c \nrm{u}{F^{s+d}_{p,q}}$;
cf Lemma~\ref{Fspq-lem}.
This gives \eqref{apsi123-eq} with $q$ replaced by $r$ on 
the left-hand side.

Thus $a_{\psi}(x,D)\colon F^{s+d}_{p,q}\to F^{s}_{p,r}$ is continuous and
coincides with $a(x,D)$ on $\cal S$. 
Since $\cal S$ is dense in $F^{s}_{p,q}$ for $q<\infty $ (and
$F^{s}_{p,\infty }\hookrightarrow F^{s'}_{p,1}$ for $s'<s$), there is no
dependence on $\psi$. Hence $u\in D(a(x,D))$ and \eqref{apsi123-eq}
holds for $a(x,D)u$. This proves \eqref{Fspr-eq}
in all cases.

The Besov case is analogous; one can interchange the order of $L_p$ and
$\ell_q$ and refer to the maximal inequality for scalar functions:
Lemma~\ref{Fspq-lem} carries over to $B^{s}_{p,q}$ in a natural way for
$0<p\le \infty $ with $r=q$ in all cases; this is well known, cf
\cite{Y1,JJ05DTL,JoSi08}. 
\end{proof}

One may also obtain \eqref{Bspq-eq}
by real interpolation of \eqref{Fspr-eq}, cf \cite[2.4.2]{T2},
when $0<p<\infty $.

The borderline analysis in \eqref{Fdp1-eq} is a little simpler than the
above, as completeness of $L_p$ may replace the use of Lemma~\ref{Fspq-lem}. 
In fact, the proof in \cite{JJ04DCR,JJ05DTL} applies to 
Definition~\ref{a11-defn} with the addendum that the right-hand side of
\eqref{a123-eq} does not depend on $\psi$ for $u\in F^d_{p,1}$,
because $\cal S$ is dense there. 

By duality, Theorem~\ref{FBspq-thm} extends to
operators that merely fulfil the twisted diagonal condition of arbitrary
real order.

\begin{thm}
  \label{FB8-thm}
Let $a(x,\eta)$ belong to the self-adjoint subclass
$\tilde S^d_{1,1}(\Rn\times\Rn)$, characterised in 
Theorem~\ref{a*-thm}. Then $a(x,D)$ 
is a bounded map for all $s\in\R$,
\begin{align}
   a(x,D)&\colon F^{s+d}_{p,q}(\Rn)\to F^{s}_{p,q}(\Rn),
   \quad 1<p<\infty,\ 1<q\le \infty, 
  \label{F8'-eq} \\
   a(x,D)&\colon B^{s+d}_{p,q}(\Rn)\to B^{s}_{p,q}(\Rn),
   \quad 1<p\le \infty,\ 1<q\le \infty.
  \label{B8'-eq}
\end{align}
\end{thm}
\begin{proof}
When $p'+p=p'p$ and $q'+q=q'q$, then $F^{s}_{p,q}$ is the dual of 
$F^{-s}_{p',q'}$ since $1<p'<\infty $ and $1\le q'<\infty $; 
cf \cite[2.11]{T2}, the case $q'=1$ is covered by eg \cite[Rem.~5.14]{FJ2}. 
The adjoint symbol $a^*(x,\eta)$ is in $S^{d}_{1,1}$ by assumption, 
and $p'\ge 1$ and $q'\ge 1$, so Theorem~\ref{FBspq-thm} gives that
\begin{equation}
  a^*(x,D)\colon F^{-s}_{p',q'}(\Rn)\to F^{-s-d}_{p',q'}(\Rn)  
\end{equation}
is continuous whenever $-s-d>\max(0,\fracc n{p'}-n)=0$, ie
for $s<-d$. The adjoint 
$a^*(x,D)^*$ is therefore bounded $F^{s+d}_{p,q}\to F^{s}_{p,q}$, and it
is a restriction of $a(x,D)$ in view of Theorem~\ref{tildeS-thm}.
When $s>0$ then \eqref{F8'-eq} also holds by Theorem~\ref{FBspq-thm}.

If $d\ge 0$ the gap with $s\in [-d,0]$ 
can be closed since $a(x,D)=b(x,D)\Lambda^t$ by Proposition~\ref{abc-prop} 
holds with $\Lambda^t=\OP((1+|\eta|^2)^{t/2})$, $t\in \R$ and 
$b(x,\eta)=a(x,\eta)(1+|\eta|^2)^{-t/2}$. The latter is of type $1,1$ and order
$-1$ for $t=d+1$, which by the just shown gives \eqref{F8'-eq} for all $s$.

For the $B^{s}_{p,q}$ scale similar arguments apply, also for $p=\infty $.
\end{proof}

Obviously Theorem~\ref{FB8-thm} gives a natural generalisation of 
H{\"o}rmander's boundedness result mentioned after \eqref{Hssigma-eq}
to the $L_p$-setting.
Specialisation of Theorems~\ref{FBspq-thm}--\ref{FB8-thm}
to Sobolev and H{\"o}lder--Zygmund spaces, cf \eqref{HF-eq}--\eqref{CB-eq}, gives
\begin{cor}
  \label{HC-cor}
Every
$a(x,D)\in \OP(S^d_{1,1}(\Rn\times\Rn))$ is bounded 
\begin{align}
   a(x,D)&\colon H^{s+d}_{p}(\Rn)\to H^{s}_{p}(\Rn),
   \quad s>0,\ 1<p<\infty,
  \label{Hsp-eq} \\
   a(x,D)&\colon C_*^{s+d}(\Rn)\to C_*^{s}(\Rn),
   \quad s>0.
  \label{Cs-eq}
\end{align}
This is valid for all real $s$ whenever $a(x,\eta)$ belongs to the
self-adjoint subclass $\tilde S^d_{1,1}(\Rn\times \Rn)$. 
\end{cor}

Previously \emph{extensions} with similar properties were obtained 
for $H^s_p$ by Meyer \cite{Mey81} and for $C^s_*$ by Stein (published in \cite{Ste93}).
By constrast, the corollary is valid for the operators in Definition~\ref{a11-defn}.

\subsection{Direct estimates for the self-adjoint subclass}
  \label{sigmaLp-ssect}
To complement Theorem~\ref{FB8-thm} with similar
results valid for $p$, $q$ in $\,]0,1]$ one can exploit 
the paradifferential decomposition \eqref{a123-eq} and the pointwise
estimates used above. 

However, in the results below there will be an
arbitrarily small loss of smoothness.
The reason is that the estimates of
$a^{(2)}_{\psi}(x,D)$ are based on a corona condition which is
\emph{non-symmetric} in the sense that the outer radii grow faster than the
inner ones.
That is, the last part of Lemma~\ref{Fspq-lem} will now be extended to
series $\sum u_j$ fulfilling 
the more general condition, where $0<\theta\le 1$ and $A>1$,
\begin{equation}
\begin{aligned}
  \supp \cal F u_j&\subset \{\,\xi\mid  |\xi|\le A 2^{j} \,\}
 \quad\text{for all $j\ge 0$},
\\
    \supp \cal F u_j&\subset \{\,\xi\mid  \tfrac{1}{A}2^{\theta j}
   \le |\xi|\le A 2^{j} \,\}
 \quad\text{for $j\ge J\ge 1$}.
\end{aligned}  
  \label{theta01-eq}
\end{equation}
This situation is probably known to experts in function spaces, 
but in lack of a reference it is analysed here. The techniques should be
standard, so the explanations will be brief.

The main point of \eqref{theta01-eq} is that $\sum u_j$
still converges for $s\le 0$, albeit with a loss of smoothness; 
cf the cases below with $s'<s$.  
Actually the loss is proportional to $(1-\theta)/\theta$, 
hence tends to $\infty $ for $\theta\to0$, which reflects that convergence 
in some cases fails for $\theta=0$ 
(take $\hat u_j=\tfrac{1}{j}\psi\in C^\infty_0$, $s=0$, $1<q\le \infty $).

\begin{prop}
  \label{Fcor-prop}
Let $s\in \R$, $0< p<\infty$, $0<q\le\infty$, $J\in\N$ 
and $0<\theta\le 1$ be given; with $q>n/(n+s)$ if $s>0$.
For each sequence
$(u_j)_{j\in \N_0}$ in $\cal S'(\Rn)$ fulfilling the corona
condition \eqref{theta01-eq}
together with the bound (usual modification for $q=\infty $)
\begin{equation}
  F:=
  \Nrm{(\sum_{j=0}^\infty |2^{sj}u_j(\cdot )|^q)^{\frac1q}}
       {L_p}<\infty,
  \label{F-id}
\end{equation}
the series $\sum_{j=0}^\infty u_j$ converges in $\cal S'(\Rn)$ to some
$u\in F^{s'}_{p,q}(\Rn)$ with
\begin{equation}
  \Nrm{u}{F^{s'}_{p,q}}\le cF,
\end{equation}
whereby the constant $c$ also depends on $s'$, which one can take as $s'=s$
for $\theta=1$, or in case $0<\theta<1$, take to fulfil
\begin{align}
  s'&=s \quad\text{for}\quad s>\max(0,\fracnp-n),   
  \label{s's-eq}
\\
  s'&<s/\theta \quad\text{for}\quad s\le 0,\; p\ge 1,\; q\ge 1,
  \label{s's-ineq}\\[-1\jot]
\intertext{or in general}
  s'&<s-\smash[t]{\tfrac{1-\theta}\theta}(\max(0,\fracnp-n)-s)_{+}.
  \label{gs's-ineq}
\end{align}
(Here $s'=s$ is possible by \eqref{s's-eq} if the positive part 
$(\dots )_+$ has strictly negative argument.)

The conclusions carry over to $B^{s'}_{p,q}$ for any $q\in ]0,\infty ]$ when
$B:=(\sum_{j=0}^\infty 2^{sjq}\nrm{u_j}{p}^{q})^{\fracci 1q}<\infty $.
\end{prop}

\begin{rem}   \label{r-rem}
The above restriction $q> n/(n+s)$ for $s>0$ is not
severe, for if \eqref{F-id} holds for a sum-exponent in $\,]0,n/(n+s)]$,
then the constant $F$ is also finite for any $q>n/(n+s)$, which yields 
the convergence and an estimate in a slightly larger space; cf the $r$ in 
Lemma~\ref{Fspq-lem}
\end{rem}

\begin{proof} Increasing $A\ge 1$, as we may, gives a reduction to
the case $J=1$: $u=\sum u_j$ has the contributions $0+\dots +0+u_J+u_{J+1}+\dots $
and $(u_0+\dots +u_{J-1})+0+\dots $, where
the former fulfils the conditions for $J=1$; the latter trivially
converges, it fulfils 
\eqref{theta01-eq} for $J=1$ if $A$ is replaced by $A2^J$ and
\eqref{F-id} as 
$\nrm{u_0+\dots +u_{J-1}}{p}\le c_p2^{|s|J}F<\infty $ with $c_p=J^{\max(1,1/p)}$.
Hence $\nrm{u}{F^{s'}_{p,q}}\le C(c+c_p2^{|s|J})F$ if $C$ is the constant
from the quasi-triangle inequality.

It is first assumed that $u=\sum u_k$ converges in 
$\cal S'$. Then each term $\Phi_j(D)\sum u_k$ in the expression for
$\nrm{u}{F^{s'}_{p,q}}$ is defined; cf \eqref{fspq-eq}. Writing now
$\Phi_j(\eta)$ as $\Phi(2^{-j}\eta)$ for clarity, one has
\begin{equation}
  \Phi(2^{-j}D)\sum_{k\ge 0} u_k=  \sum_{j-h\le k\le [j/\theta]+h} 
  \Phi(2^{-j}D)u_k.
  \label{Phiksum-eq}
\end{equation}
In fact, \eqref{theta01-eq}  gives an $h\in \N$ such that
$\Phi(2^{-j}D)\cal F u_k=0$ for all $k\notin[j-h,\frac{j}{\theta}+h]$.

To proceed it is convenient to use Marschall's inequality;
cf Remark~\ref{Marschall-rem}.  This gives
\begin{equation}
  |\Phi(2^{-j}D)u_k(x)|\le c
  \Nrm{\Phi(R2^{\nu-j}\cdot )}{\dot B^{\frac{n}{t}}_{1,t}}
  M_t u_k(x), \quad\text{for}\quad 0<t\le 1,
\end{equation}
whereby $\nu$ should be taken so large that $B(0,R2^\nu)$
contains the supports of $\Phi(2^{-j}\cdot )$ and $\hat u_k$; 
also $R\ge A$ can be arranged.
Note that by Remark~\ref{Marschall-rem},
\begin{equation}
  \Nrm{\Phi(R2^{\nu-j}\cdot )}{\dot B^{\frac{n}{t}}_{1,t}}
  =
  2^{(\nu-j)(\frac{n}{t}-n)}\Nrm{\Phi(R\cdot )}{\dot B^{\frac{n}{t}}_{1,t}}.
\end{equation}
This is applied in the following for some $t\in \,]0,1]$ 
that also fulfils $t<\min(p,q)$, and the main point is to show that,
with $s'$ as in the statement, it holds in all cases that
\begin{equation}
 \big (\sum_{j=0}^\infty 2^{s'jq}|\Phi(2^{-j}D)\sum_{k\ge 0}u_k(x)|^q
 \big)^{1/q}
\le 
  c \big (\sum_{k=0}^\infty 2^{skq} M_t u_k(x)^q
  \big)^{1/q}.
  \label{prel-ineq}
\end{equation}

The easiest case is for $0<q\le 1$.
As $\ell_q\hookrightarrow \ell_1$ for such $q$, one has
\begin{equation}
  \begin{split}
  \sum_{j=0}^\infty 2^{s'jq}|\Phi(2^{-j}D)\sum_{k\ge 0}u_k(x)|^q
&\le 
  \sum_{j=0}^\infty\sum_{j-h\le k\le j/\theta+h} 
    2^{s'jq}|\Phi(2^{-j}D)u_k(x)|^q
\\
&\le 
  c\sum_{k=0}^\infty\sum_{\theta k-h\le j\le k+h} 2^{s'jq}
    \Nrm{\Phi(R2^{\nu-j}\cdot )}{\dot B^{\frac{n}{t}}_{1,t}}^q 
    M_tu_k(x)^q. 
  \end{split}
  \label{2sum-ineq}
\end{equation}
Here $\nu=j$ gives a constant for $j\ge k$, so the above is both for 
$s'\gtreqless 0$ estimated by
\begin{equation}
  c\sum_{k=0}^\infty(
  h2^{s'kq}+ \sum_{\theta k-h\le j\le k} 2^{s'jq+(\frac{n}{t}-n)(k-j)q}
    )M_tu_k(x)^q.
\end{equation}
For $\theta=1$ the sum over $j$ has a fixed number of terms, hence is 
$\cal O(2^{skq})$ for $s'=s$; cf \eqref{prel-ineq}.

In the case in \eqref{s's-eq} one may as $q>n/(n+s)$ arrange that
$s'=s>\tfrac{n}{t}-n>\max(0,\fracc np-n,\fracc nq-n)$ by taking $t$
sufficiently close to $\min(p,q)$. 
Then the geometric series above is estimated by the last term, 
hence is $\cal O(2^{skq})$, as required in \eqref{prel-ineq}.

What remains of  \eqref{gs's-ineq} are the cases in which $s\le
\max(0,\fracnp-n)$, that is
\begin{equation}
  s'<s\le \max(0,\fracc np-n,\fracc nq-n)<\tfrac{n}{t}-n,
\qquad t\in \,]0,\min(p,q)[\,.
\end{equation}
By \eqref{gs's-ineq} a suitably small $t>0$ yields $s=\theta s'+(1-\theta)(\frac
nt-n)$, and since
$s'-(\frac{n}{t}-n)<0$ in the above sum an estimate by the first term gives 
$\cal O(2^{(s'\theta+(1-\theta)(\frac{n}{t}-n))kq})= \cal O(2^{skq})$.

For $1<q<\infty $ the inequality \eqref{prel-ineq} follows by use of H{\"o}lder's inequality
in \eqref{Phiksum-eq}, for if $q+q'=q'q$, one can for $s'<0$ use 
$2^{\theta s'(k-j)}$ as a summation factor to get
\begin{equation}
  |\Phi(2^{-j}D)\sum_{k\ge 0}u_k(x)|^q
\le 
  c\sum_{k=j-h}^{[j/\theta]+h}
  2^{(k-j)s'\theta q}\Nrm{\Phi(R2^{\nu-j}\cdot )}{\dot B^{\frac nt}_{1,t}}^q
  M_t u_k(x)^q
  (\frac{2^{-(\tfrac{1}{\theta}-1)js'\theta q'}}
        {2^{-s'\theta q'}-1})^{\tfrac{q}{q'}}.
\end{equation}
Therefore the above procedure yields an estimate of
$\sum_{j=0}^\infty 2^{s'jq}|\Phi(2^{-j}D)\sum_{k\ge 0}u_k(x)|^q$ by
\begin{equation}
  \sum_{k=0}^\infty 
   2^{ks'\theta q} M_t u_k(x)^q(h+
  \sum_{\theta k-h\le j< k} 
  2^{(k-j)(\frac nt-n)q})
\le 
  c \sum_{k=0}^\infty2^{(s'\theta+(1-\theta)(\tfrac{n}{t}-n))kq} 
     M_tu_k(x)^q,
\end{equation}
which again gives \eqref{prel-ineq}
by using \eqref{gs's-ineq} to arrange $s\ge s'\theta+(1-\theta)(\tfrac{n}{t}-n)$ for a $t\in\,]0,1[\,$. 
By making the last inequality strict for a slightly larger $t$, the
argument is seen to extend to cases with $0\le s'<s\le \max(0,\fracc np-n)$ 
by using $s'-(\frac nt-n)<0$ instead of $s'$ in H{\"o}lder's inequality.
In fact, one gets 
$\sum 2^{(s'\theta +(1-\theta)(\frac nt-n))kq}(h2^{h(\frac nt-n)}+(1+h+k(1-\theta)))M_tu_k(x)^q$, 
which again is $\cal O(2^{skq})$ as the term $k(1-\theta)$ is harmless 
by the choice of $t$ (or for $\theta=1$). Hence \eqref{prel-ineq} holds.

In case $s'=s>0$, cf \eqref{s's-eq}, one may take $s-\frac nt+n>0$ (as for
$q\le 1$) now with
$2^{(k-j)(s-\frac nt+n)/2}$ as a summation factor: then $(\dots)^{q/q'}=\cal O(1)$, 
so the factor in front of
$M_tu_k^q$ becomes
\begin{equation}
  \sum_{\theta k-h\le j\le k+h} 2^{sjq+(k-j)(s-\frac nt+n)q/2
         +(k-j)_+(\frac nt-n)q}
  =\cal O(2^{skq}).
\end{equation}
For $q=\infty $ a direct argument yields sup-norms weighted by
$2^{s'j}$ and $2^{sk}$ in \eqref{prel-ineq}.

By the choice of $t$, the Fefferman--Stein inequality applies to
\eqref{prel-ineq}, cf \eqref{u*M-eq}, whence
\begin{equation}
  (\int_{\Rn}(
  \sum_{j=0}^\infty 2^{s'jq}|\Phi_j(D)\sum_{k\ge 0}u_k(x)|^q)^{p/q}
  \,dx)^{1/p}
\le c
   (\int\Nrm{2^{sk}u_k(\cdot )}{\ell_q}^p\,dx)^{1/p}
=c F.
  \label{cF-ineq}
\end{equation}

Convergence is trivial for the partial sums $u^{(m)}=\sum_{j\le m}u_j$,
hence for $u^{(m+M)}-u^{(m)}$. 
So \eqref{cF-ineq} 
applies to $(0,\dots 0,u_{m+1},\dots ,u_{m+M},0,\dots )$,
which for $q<\infty $ by majorisation for $m\to\infty $ yields
\begin{equation}
  \Nrm{u^{(m+M)}-u^{(m)}}{F^{s'}_{p,q}}
\le c
   (\int_{\Rn}(\sum_{k=m}^\infty 2^{skq}|u_k(x)|^q)^{p/q}\,dx)^{1/p}
\searrow 0.
\end{equation}
As $F^{s'}_{p,q}$ is complete, $\sum u_j$ converges to an element $u(x)$ 
with norm 
$\le c F$ according to \eqref{cF-ineq}. For $q=\infty $ there is convergence
in the larger space $F^{s'-1/\theta}_{p,1}$ since the constant
$F$ remains finite if $s$,$\infty $ are replaced by $s-1$, $1$;
and again $\nrm{u}{F^{s'}_{p,q}}\le cF$ holds by \eqref{cF-ineq}.

For the Besov case the arguments are analogous. First of all the absolute
value should be replaced by the norm of $L_p$  in \eqref{2sum-ineq}, 
that now pertains to $0<q\le \min(1,p)$. H{\"o}lder's inequality applies in this case
if $1/q+1/q'=1/\min(1,p)$; and \eqref{cF-ineq} can be replaced by
boundedness of $M_t$ in $L_{p}$ for $t<p$. Convergence is
similarly shown. 
\end{proof}

Thus prepared, one arrives at a general result for $0<p\le 1$.

\begin{thm}
  \label{FB8'-thm}
If $a(x,\eta)$ belongs to the self-adjoint subclass
$\tilde S^d_{1,1}(\Rn\times\Rn)$, the operator $a(x,D)$ 
is bounded for $0<p\le 1$, $0<q\le\infty$, 
\begin{align}
   a(x,D)&\colon F^{s+d}_{p,q}(\Rn)\to F^{s'}_{p,q}(\Rn)
   \quad \text{for } s'<s\le\fracnp-n,
  \label{F8-eq} \\
   a(x,D)&\colon B^{s+d}_{p,q}(\Rn)\to B^{s'}_{p,q}(\Rn)
   \quad\text{for } s'<s\le \fracnp-n.
  \label{B8-eq}
\end{align}
\end{thm}

\begin{proof}
Using \eqref{chi1-eq}--\eqref{chi3-eq}, the question is easily reduced to the
case of symbols for which 
\begin{equation}
 \hat a(\xi,\eta)\ne0  \implies  \max(1,|\xi+\eta|)\le |\eta|.
  \label{ahat-eq}
\end{equation}
In fact $a=a_{\chi,1}+(a-a_{\chi,1})$ where $a_{\chi,1}$ has the above
property, whilst Theorem~\ref{FBspq-thm} yields the boundedness for
$a-a_{\chi,1}$, as this is easily seen to fulfil the twisted diagonal
condition \eqref{tdc-cnd} for $B=1$. 
(Note that $a-a_{\chi,1}\in \tilde S^d_{1,1}$ is seen from Theorem~\ref{a*-thm},
as in \eqref{a*tdc-eq}, so that also $a_{\chi,1}\in\tilde S^d_{1,1}$.) 

First $a^{(1)}(x,D)u$ and $a^{(3)}(x,D)u$ are for all
$s\in \R$ covered by the proof of Theorem~\ref{FBspq-thm}; cf \eqref{apsi123-eq}. 
Thus it suffices to estimate the
$a^{(2)}$-series in \eqref{a2'-eq} for fixed $s'<s\le \fracnp-n$; 
a simple embedding of $F^{s'}_{p,q}$ gives a reduction to the case $q>n/(n+s)$ if
$s>0$; cf also Remark~\ref{r-rem}.

To fix notation, the splitting \eqref{a123-eq} is considered for some modulation
function $\Psi$ for which the associated Littlewood--Paley decomposition 
$1=\sum \Phi_j$ is used in the definition of the norms on $F^s_{p,q}$, as
described prior to \eqref{fspq-eq}.
Subjecting the second term in \eqref{a2'-eq} to H{\"o}rmander's localisation to a
neighbourhood of $\cal T$, 
cf \eqref{chi1-eq}--\eqref{chi3-eq}, one arrives at
\begin{equation}
  \hat a_{k,\chi,\varepsilon}(\xi,\eta)=
  \hat a(\xi,\eta)\Phi(2^{-k}\xi)\chi(\xi+\eta,\varepsilon\eta),
  \label{akke-id}
\end{equation}
This leaves the remainder
$b_k(x,\eta)=a_k(x,\eta)-a_{k,\chi,\varepsilon}(x,\eta)$,
that applied to the difference
$v_k=u^{k-1}-u^{k-h}
  =\cal F^{-1}((\Phi(2^{1-k}\cdot )-\Phi(2^{h-k}\cdot))\hat u)$
in \eqref{a2'-eq} gives
\begin{equation}
  a_k(x,D)v_k=a_{k,\chi,\varepsilon}(x,D)v_k+b_k(x,D)v_k .
\end{equation}

To utilise the pointwise estimates, take $\psi\in
C^\infty_0(\Rn)$ equal to $1$ around the corona given by 
$\tfrac{r}{R}2^{-1-h}\le |\eta|\le 1$ and supported 
where $\tfrac{r}{R}2^{-2-h}\le |\eta|\le 2$. 
Using $\psi(\eta/(R2^k))$ as the auxiliary function  
in the symbol factor,
the factorisation inequality \eqref{Fau*-eq} and Theorem~\ref{Fa-thm} give
\begin{equation}
  \begin{split}
  |a_{k,\chi,\varepsilon}(x,D)v_k(x)|&\le 
  F_{a_{k,\chi,\varepsilon}}(N,R2^k;x)v^*_k(N,R2^k;x)
\\
  &\le 
  cv^*_k(x)\sum_{|\alpha|=0}^{[N+n/2]+1}
  (\int_{r2^{k-h-2}\le |\eta|\le R 2^{k+1}} |(R2^k)^{|\alpha|-n/2}
   D^\alpha_\eta a_{k,\chi,\varepsilon}(x,\eta)|^2\,d\eta)^{1/2}.
  \end{split}
\end{equation}
Here the ratio of the limits is $2R/(r2^{-h-2})>32$, 
so the integration can be extended to $L\ge 6$ dyadic coronas,
with $|\eta|\in [R2^{k+1-L},R2^{k+1}]$.
This gives an estimate by $c(R2^k)^d L^{1/2}N_{\chi,\varepsilon,\alpha}(a_k)$. 
In addition, Minkowski's inequality gives
\begin{equation}
  N_{\chi,\varepsilon,\alpha}(a_k)\le 
  \sup_{\rho>0}\rho^{|\alpha|-d}\int_{\Rn} |2^{kn}\check \Phi(2^{k}y)|
   (\int_{\rho\le |\eta|\le 2\rho}    |D^\alpha_\eta 
   a_{\chi,\varepsilon}(x-y,\eta)|^2\,\frac{d\eta}{\rho^n})^{1/2}\,dy
 \le c
  N_{\chi,\varepsilon,\alpha}(a).
\end{equation} 
So  it follows from the above and \rom{(iii)} in Theorem~\ref{a*-thm} that
for all $\sigma>0$,
\begin{equation}
  |a_{k,\chi,\varepsilon}(x,D)v_k(x)|\le  cv^*_k(N,R2^k;x) 2^{k+d}
  \sum_{|\alpha|\le [N+n/2]+1}c_{\alpha,\sigma}
  \varepsilon^{\sigma+n/2-|\alpha|}.
   \label{akkev-eq}
\end{equation}

Now $\theta\in \,]0,1[\,$ is taken so small that 
$s'< s-\frac{\theta}{1-\theta}(\fracnp-n-s)$, which is the last condition in
Proposition~\ref{Fcor-prop} with $1-\theta$ instead of $\theta$.
Then $\varepsilon=2^{-k\theta}$ 
in \eqref{akkev-eq} clearly gives
\begin{equation}
  2^{k(s+M)}|a_{k,\chi,\varepsilon}(x,D)v_k(x)| 
\le cv_k^*(N,R2^k;x)2^{k(s+d)} 2^{-k\theta(\sigma-1-N-M/\theta)}.
  \label{akkev''-eq}
\end{equation}
Here one may first of all take $N>n/\min(p,q)$ so that \eqref{u*M-eq}
applies. Secondly, 
$\sigma$ can for any $M$ (with $\theta$ fixed as above) 
be chosen so that $2^{-k\theta(\sigma-1-N-M/\theta)}\le 1$.
This gives
\begin{equation}
  \begin{split}
  (\int\Nrm{2^{k(s+M)}a_{k,\chi,\varepsilon}(x,D)v_k(\cdot)}{\ell_q}^p
    \,dx)^{\fracpi}
  &\le c (\int\Nrm{2^{k(s+d)}v_k^*(N,R2^k;\cdot )}{\ell_q}^p\,dx)^{\fracpi}
\\
  &\le c' (\int\Nrm{2^{k(s+d)}v_k(\cdot )}{\ell_q}^p\,dx)^{\fracpi}
\le c''\Nrm{u}{F^{s+d}_{p,q}}.
  \label{akkeLp-eq}
  \end{split}
\end{equation}
Here the last inequality follows from the (quasi-)triangle inequality in
$\ell_q$ and $L_p$.

Since $a_{k,\chi,\varepsilon}(x,D)v_k$ according to \eqref{supp2-eq} has its spectrum in 
$\overline{B}(0,2R2^k)$, the above estimate allows application of
Lemma~\ref{Fspq-lem}, if $M$ is so large that 
\begin{equation}
  M>0,\quad M+s>0,\quad M+s>\fracnp-n.
  \label{M-ineq}
\end{equation}
This gives convergence of 
$\sum a_{k,\chi,2^{-k\theta}}(x,D)v_k$ to a function in
$F^{s+M}_{p,\infty }$ fulfilling
\begin{equation}
  \Nrm{\sum_{k=1}^\infty a_{k,\chi,2^{-k\theta}}(x,D)v_k}{F^{s+M}_{p,\infty }}
  \le c\nrm{u}{F^{s+d}_{p,q}}.
  \label{a21-eq}
\end{equation}
On the left-hand side the embedding $F^{s+M}_{p,\infty }\hookrightarrow
F^{s}_{p,q}$ applies, of course.

For the remainder $\sum_{k=1}^\infty b_k(x,D)v_k$, cf \eqref{akke-id}
ff, note that \eqref{akkeLp-eq} holds for $M=0$ with the same $\sigma$. 
If combined with a part of \eqref{a2Lplq-eq}, a crude use of the (quasi-)triangle
inequality gives 
\begin{equation}
  \int\Nrm{2^{ks}b_k(x,D)v_k(\cdot)}{\ell_q}^p
    \,dx
\le \int\Nrm{2^{ks}(a_k(x,D)-a_{k,\chi,2^{-k\theta}}(x,D))
        v_k(\cdot)}{\ell_q}^p \,dx
 \le c \nrm{u}{F^{s+d}_{p,q}}^p.
\end{equation} 
The series also fulfils a corona condition
with inner radius $2^{(1-\theta)k}$ for all large $k$, namely
\begin{equation}
  \supp\cal Fb_k(x,D)v_k\subset
  \Set{\zeta}{(r2^{-h-2})2^{k(1-\theta)}\le |\zeta|\le R2^k}.
  \label{bkcorona-eq}
\end{equation}
Indeed, $\hat b_k(x,\eta)=0$ holds if $\chi(\xi+\eta,2^{-k\theta}\eta)=1$, 
so at least for $2\max(1,|\xi+\eta|)\le  2^{-k\theta}|\eta|$;
whence by \eqref{ahat-eq},
\begin{equation}
  \supp \hat b_k\subset \Set{(\xi,\eta)}{2^{-1-k\theta}|\eta|\le 
  \max(1,|\xi+\eta|)\le |\eta|}.
\end{equation}
The Spectral Support Rule \eqref{Sigma-eq} shows that $\zeta=\xi+\eta$ only belongs to
$\supp\cal Fb_k(x,D)v_k$ if
\begin{gather}
  |\zeta|\le |\eta|\le R2^k
\\
  \max(1,|\zeta|)\ge 2^{-1-k\theta}|\eta|\ge r2^{k(1-\theta)-h-2}.  
\end{gather}
When $2^{k(1-\theta)}>2^{h+2}/r$ 
(so that the last right-hand side is $>1$) this shows \eqref{bkcorona-eq}.
Hence Proposition~\ref{Fcor-prop} applies, and the choice of $\theta$ gives
\begin{equation}
  \Nrm{\sum_{k=1}^\infty b_k(x,D)v_k}{F^{s'}_{p,q }}
  \le c\nrm{u}{F^{s+d}_{p,q}}.  
  \label{a22-eq}
\end{equation}

The other contribution $\sum
(a^k(x,D)-a^{k-h}(x,D))u_k$ in \eqref{a2'-eq} is analogous, with a splitting 
of $\tilde a_k=a^k-a^{k-h}$ into $\tilde a_{k,\chi,\varepsilon}+\tilde b_k$ as in 
\eqref{akke-id}.
In particular the inequality \eqref{akkev-eq} can be carried over to
$\tilde a_{k,\chi,\varepsilon}(x,D)u_k$, 
with just another constant because Minkowski's inequality now leads to
an estimate in terms of $\int |\Psi-\Psi(2^h\cdot)|dy$ . Consequently
\eqref{akkev''-eq} carries over, and with \eqref{M-ineq} the same
arguments as for \eqref{a21-eq}, \eqref{a22-eq} give 
\begin{equation}
  \Nrm{\sum_{k=h}^\infty 
        (a^k-a^{k-h})_{\chi,\varepsilon}(x,D)u_k}{F^{s+M}_{p,\infty }}
+  \Nrm{\sum_{k=h}^\infty \tilde b_k(x,D)u_k}{F^{s'}_{p,q }}
  \le c\nrm{u}{F^{s+d}_{p,q}}.  
  \label{a23-eq}
\end{equation}
Altogether the estimates \eqref{a21-eq}, \eqref{a22-eq}, \eqref{a23-eq}
show that
\begin{equation}
  \Nrm{a^{(2)}_{\psi}(x,D)u}{F^{s'}_{p,q}}\le c\Nrm{u}{F^{s+d}_{p,q}}.
\end{equation}
Via the decomposition \eqref{a123-eq}, 
$a_{\psi}(x,D)$ is therefore a bounded linear map $F^{s+d}_{p,q}\to
F^{s'}_{p,q}$. Since $\cal S$ is dense for $q<\infty $ (a case one
can reduce to), there is no dependence on the 
modulation function $\psi$, so the type $1,1$-operator
$a(x,D)$ is defined and continuous on $F^{s+d}_{p,q}$ as stated.

The arguments are similar for the Besov spaces: it suffices to interchange
the order of the norms in $\ell_q$ and $L_p$, and to use the estimate in 
\eqref{u*M-eq} for each single $k$. 
\end{proof}

The proof extends to cases with $0<p\le \infty $ when $s'<s\le \max(0,\frac np-n)$,
but this barely fails to reprove Theorem~\ref{FB8-thm}, so only $p\le 1$ is
included in Theorem~\ref{FB8'-thm}. Cf also Remark~\ref{lp-rem} below.

One particular interest of Theorem~\ref{FB8'-thm} is that $F^0_{p,2}(\Rn)$
identifies with the
so-called local Hardy space $h_p(\Rn)$ for $0<p\le 1$; cf \cite{T2} and
especially \cite[Ch.~1.4]{T3}. In this case Theorem~\ref{FB8'-thm} gives
boundedness as a map $a(x,D)\colon h_p(\Rn)\to F^{s'}_{p,2}(\Rn)$ for every $s'<0$,
but this can probably be improved in view of recent results:

\begin{rem} \label{lp-rem}
Extensions to  $h_p(\Rn)$ of operators in the self-adjoint subclass 
$\OP(\tilde S^0_{1,1})$
were treated by Hounie and dos Santos Kapp \cite{HoSK09}, who
used atomic estimates to carry over the
$L_2$-boundedness of H{\"o}rmander \cite{H89,H97} to $h_p$, ie to obtain
estimates with $s'=s=0$. However, they
worked without a precise definition of type $1,1$-operators. 
Torres~\cite{Tor90} obtained extensions by continuity 
using the atomic decompositions in \cite{FJ2}, but 
for $s<0$ he relied on conditions on the adjoint $a(x,D)^*$ rather
than on the symbol $a(x,\eta)$ itself.
In the $F^{s}_{p,q}$-scales, general type $1,1$-operators were
first estimated by Runst~\cite{Run85ex}, though with insufficient control of
the spectra as noted in \cite{JJ05DTL};
a remedy is provided by the Spectral Support Rule \eqref{Sigma-eq}.
\end{rem}

\begin{rem}
Together Theorems~\ref{FBspq-thm}, \ref{FB8-thm}
and \ref{FB8'-thm} give a satisfactory $L_p$-theory of operators
$a(x,D)$ in the self-adjoint subclass $\OP(\tilde S^{d}_{1,1})$, inasmuch as for the domain
$D(a(x,D))$ they cover all possible $s$, $p$. 
Only a few of the codomains seem barely
unoptimal, and these all concern cases with $0<q<1$ or $0<p\le 1$;
cf the role of the parameter $r$ in Theorem~\ref{FBspq-thm} and that
of $s'$ in Theorem~\ref{FB8'-thm}.
\end{rem}

\begin{rem}
As a corollary to Theorem~\ref{FB8'-thm}, its proof (extended to
$p\ge1$) gives that if $a(x,D)$ fulfils the
twisted diagonal condition of order $\sigma>0$, 
i.e.\ \eqref{Hsigma-eq} holds for a specific $\sigma$, then it is not difficult to see that
\begin{equation}
 B^{s}_{p,q}\bigcup F^{s}_{p,q}\subset D(a(x,D))  
\quad\text{for}\quad s>-\sigma+[N+n/2]+1-n/2,\quad 1\le p\le \infty. 
\end{equation}
Hereby $N>n/p$ must hold (as $q=\infty$ suffices now), so the condition has the form 
$s>-\sigma+k$, where $k=[n/p]+1$ in even dimensions, while in odd dimensions $k$ should be the least number in $\tfrac12+\N_0$ such that $k>n/p$. 
While this does provide a result
in the $L_p$ set-up, it is hardly optimal; cf H{\"o}rmander's condition
$s>-\sigma$ for $p=2$, recalled in \eqref{Hssigma-eq}.
\end{rem}

%
\providecommand{\bysame}{\leavevmode\hbox to3em{\hrulefill}\thinspace}

\end{document}